\documentclass[12pt]{article}

\usepackage{xcolor}
\usepackage{amssymb}
\usepackage{mathrsfs}

\usepackage{common}

\title{Stability over a predicate and prime closure}

\author{Alexander Usvyatsov\footnote{The author thanks the Austrian Science Foundation (FWF), projects P33895 and  P33420 for supporting this research.
    }}

\newcommand{\Addresses}{{
  \bigskip
  \footnotesize

 \medskip

Alexander Usvyatsov, \textsc{Institut f\"{u}r Diskrete Mathematik und Geometrie,
TU Wien,
1040 Vienna, Austria}
\medskip
}}
   
\newtheorem{theorem}{Theorem}[section]
\newtheorem{definition}[theorem]{Definition}
\newtheorem{example}[theorem]{Example}
\newtheorem{lem}[theorem]{Lemma}
\newtheorem{obs}[theorem]{Observation}
\newtheorem{co}[theorem]{Corollary}

\newtheorem{hyp}[theorem]{Hypothesis}
\newtheorem{remark}[theorem]{Remark}
\newtheorem{pr}[theorem]{Proposition}

\newtheorem{conv}[theorem]{Convention}

\newtheorem{ft}[theorem]{Fact}
\renewenvironment{proof}{\noindent {\em Proof:}}{\hspace*{1cm}
        \hspace*{\fill}$\rule{1.2ex}{1.4ex}$\medskip}

\newenvironment{de}{\begin{definition}\rm}{\end{definition}}
\newenvironment{fact}{\begin{ft}\rm}{\end{ft}}

\newcommand{\red}[1]{{\color{black}{#1}}}

\date{}

\begin{document}


\maketitle

\abstract{We prove that in a theory $T$ stable over a predicate $P$, for any $\lam>|T|$, there is a $\lam$-prime model over any complete set $A$ with a $\lam$-saturated $P$-part.}

\section{Introduction}

This paper is concerned with further development of classification theory over a predicate. 

Specifically, our context is 
the following: $T$ is a complete first order theory, and $P$ is a distinguished unary predicate in the vocabulary of $T$. Our goal is understanding (and classifying) models $M \models T$ when a fixed $P^M = P^*$.  The idea is that $P$ itself can be very  complex, whereas ``over P'' the situation is nice and classifiable. 

To make things a bit more specific, let $(T,P)$ be such a pair. Let $\tau$ be the vocabulary of $T$; for simplicity, we shall assume that $\tau$ has no function symbols. Let $\cC$ be the monster model of $T$. We will denote the theory of $P$ (that is, the $\tau$-complete theory of $P^\cC$, which can be thought of as a substructure of $\cC$) by $T^P$. Now, given $M \models T$, we can ask: can one classify models ``over $P^M$''? Alternatively, what can be said about the class of models $N \models T$ with $P^N = P^M$? Even when there is no hope for classification of models of $T$ in general, this problem can have a simple and clear answer. 

One natural example of such a phenomenon is the theory $T$ of vector spaces over a field $K$, where $P$ is the field. We can think of its models as two-sorted structures $M = (K,V)$, where $P^M = K$. Clearly, even when $K$ completely  model theoretically intractable (e.g., $T^P = Th(\mathbb{Q})$), the class of vector spaces over a fixed field has a nice and simple  classification. 

Much work on classification over $P$ has  already been done by Hodges, Shelah, Pillay, and others: e.g., \cite{Hod-cat1,Hod-cat2,Hod-cat3, Pil-cat, PiSh130, Sh234}, and some instructive examples have been worked out in e.g. \cite{HaSh323, KiZil, Hen-thesis}. A more recent preprint of Shelah and the author \cite{ShUs322a} develops a general structure theory. In this paper, we are going to build and expand on the techniques developed in the latter paper. 

\smallskip

One interesting specific example to consider in this setting, studied by Chatzidakis in  \cite{Chatzidakis2020RemarksAT}, is the theory of algebraically closed fields of characteristic 0 with a generic automorphism $\sigma$, where $P$ is the fixed field of $\sigma$. This example is somewhat unusual, since, in this case, $P$  is actually model theoretically ``tame'' (it is pseudo-finite), and the full theory $T$ is also very well understood (it is simple, has QE, etc). Nevertheless, considering it in this new framework, offers further insight. In fact, ``over $P$'', one obtains (the appropriate version of) stability (even superstability), which allows for stronger classification results, once $P$ is fixed. 

For example, Chatzidakis has shown (Theorem 3.14 in \cite{Chatzidakis2020RemarksAT}) the existence of a $\ka$-prime  $\ka$-atomic model over an algebraically closed difference field of characteristic $0$ with a $\ka$-saturated and pseudo-finite $P$-part (for $\ka$ uncountable or $\aleph_\eps$). In this paper, we prove a general result of this type.

\medskip

Following \cite{ShUs322a}, we study consequences of the natural assumption that $T$ is stable over $P$; that is, that 
over any ``complete'' set, there are ``few'' types orthogonal to $P$ (see Definitions \ref{dfn:complete}, \ref{dfn:startypes}). Our main goal is the existence of prime models in this context. Specifically, assume that $M \models T$. We can ask, whether in the class of models $\set{N\models T \colon P^N = P^M}$ there is a prime model. More generally, if, in addition, $P^M$ is $\lam$-saturated, we can ask the same question for the class of $\lam$-saturated models $N$ of $T$ with $P^N = P^M$. As mentioned above, Chatzitakis has shown that the answer is positive for $ACFA_0$ (over $P = Fix(\sigma)$) for $\lam > \aleph_0$ (including the case $\lam = \aleph_\eps$). This result is consistent with the classical theory of stability (one would expect a theorem along these lines for any class, considered to be ``superstable'' in some model theoretic context).  

In \cite{ShUs322a}, an analogous theorem was proven for complete sets with a saturated $P$-part. However, saturated models are hard to come by for unstable theories. Even the theory of pseudo-finite fields ($T^P$ in the case of $ACFA_0$) does not necessarily have saturated models (that is, it has a saturated model of an uncountable cardinality $\lam$ if and only if $\lam = \lam^{<\lam}$). It is therefore natural and interesting to consider the situation when $P^M$ is $\lam$-saturated, but is of (potentially) larger cardinality; for example, it is $\aleph_1$-saturated of cardinality the continuum. These cases are covered by the theory developed in \cite{Chatzidakis2020RemarksAT}   (in the case of $ACFA_0$), but can not be obtained just with the coarse methods used in  \cite{ShUs322a} (which rely on the $P$-part being truly saturated). 
%

In this paper, we provide a more nuanced analysis in order to prove a theorem analogous to the main result of \cite{Chatzidakis2020RemarksAT} in the general context. In order to do this, we need to make a stronger assumption than the one used in \cite{ShUs322a}: specifically, we assume that $T$ is ``fully'' stable over $P$. This means that all complete sets are stable over $P$, see Definitions  \ref{dfn:complete}, \ref{dfn:fullystable} (whereas in \cite{ShUs322a} it was enough to assume stability of models). However, we also obtain a much stronger result:  existence of $\lam$-prime models over complete sets $A$ with a \lam-saturated $P$-part for all $\lam>|T|$. 

We also work under the assumption that there exists a $\lam$-saturated model in the class $\set{N\models T \colon P^N = P^M}$ (otherwise, the question of existence of ``special'' models in this class becomes quite meaningless). This assumption holds for all models (and, in fact, for the class $\set{N\models T \colon P^N = P^A}$ for any complete set $A$ with a $\lam$-saturated $P^A$) in the case of $ACFA_0$ (by Lemma 3.2 in \cite{Chatzidakis2020RemarksAT}). 

Let us state a simplified version of our main result, which is easy to phrase (see Corollary \ref{cor:main}):

\begin{co}
	Let $T$ be fully stable over $P$, $A$ be a complete set (e.g., $A \models T$) with $P^A$ $\lam$-saturated for some $\lam>|T|$. Assume that the class of models 
	$$\mathcal{K} = \set{N\models T \colon P^N = P^A, N \text{ is } \lam-\text{saturated}}$$
	is non-empty. Then $\mathcal{K}$ has a prime member over $P$: that is, there exists $N_0 \in \mathcal{K}$, which is 
	elementarily embeddable into any $N \in \mathcal{K}$ over $P^A$. 
\end{co}

We intend to return to a more detailed discussion of superstability (as opposed to stability) and the case $\ka = \aleph_\eps$ in a future work.


\medskip

This paper is organized as follows. In Section 2 we set up the general context. Section 3 recalls some of the relevant definitions and results. Section 4 is devoted to proving the main results.

\subsection*{Acknowledgements} The author thanks Zo\'{e} Chatzidakis for inspiring discussions.

\section{The setting}

\begin{conv}
Let $T$ be a complete first order theory, $P$ a monadic predicate in
its vocabulary. For simplicity, we will assume that the language of $P$ does not contain function symbols (so 
that every subset of a model, containing all the constant symbols, is a substructure). 
\end{conv}

\noindent
Let $\mathcal{C} $ be the monster model of $T$. From now on, we assume that all models of $T$ are elementary submodels of $\mathcal{C}$, and all sets are subsets of $\mathcal{C}$. \medskip
 
For $M \models T$, we denote by  $M|_P$ the set $P^M$ viewed as a substructure of $M$. Similarly, for a subset $A \subseteq M$, we denote by $M|_A$ the substructure of $M$ with universe $A$. We write $A\equiv B$ if $Th({\cal C}|_A)=Th({\cal
C}|_B)$.

We also denote $T^P = Th(\cC|_P)$. For a set $A$, we denote $P^A = A \cap P^\cC$.

When no confusion should arise, we will write $P$ for $P^\cC$. Also, for a set $A$, we will often denote by $A$ both the set and the substructure of $\cC$ with universe $A$. So for example, when we write that $A\cap P^\cC$ is $\lam$-saturated, or just that $A\cap P$ is $\lam$-saturated, we mean that the substructure $\cC|_{A\cap P^\cC}$ is a $\lam$-saturated model of the appropriate theory (if $A \cap P \prec P$, which will be the case in this paper, then the appropriate theory is $T^P$).


%

\bigskip

Throughout the paper, we are going to make the following fundamental assumptions on $T$:

\begin{hyp}\label{asm:1} (\emph{\underline{Hypothesis 1})}

  
\begin{enumerate}
 \item Every type over $P^\mathcal{C}$ is definable. In other words, $P$ is \emph{stably embedded}: subsets of $P^\mathcal{C}$  that are definable (in $\mathcal{C}$) with parameters,  are definable in
$\mathcal{C}|_{P}$ with parameters from $P^\mathcal{C}$. 
\item  Moreover, subsets of $P^\mathcal{C}$  that are 0-definable (in $\mathcal{C}$),  are already  0-definable in
$\mathcal{C}|_{P}$. 
\item $T$ has quantifier elimination (even down to the level of predicates). 
 \end{enumerate}
\end{hyp}

 These assumptions are obviously very convenient, but there are also ``good'' reasons to assume them. First, they hold in many important and interesting examples. Second, it was shown by Pillay and Shelah in \cite{PiSh130} that these assumptions are in some sense ``justified'' in our context. Specifically, they have shown that one can deduce a non-structure result from the failure of Hypothesis (i) above (and in \cite{ShUs322a} it is explained why, if (i) holds, one may assume without loss of generality, that (ii) and (iii) hold as well) . Since we are interested in structure theory, we can assume that our $T$ falls on the ``good'' side of these dividing lines. See \cite{ShUs322a} for a more detailed discussion.

\section{Completeness, stability, and rank}

Let us recall some basic notions and results from \cite{PiSh130} and \cite{ShUs322a}.

\subsection{Completeness and relevant types}

In trying to reconstruct $M$ from $M| P^M$ one needs to work with sets $A$ satisfying
$P^M\subseteq A\subseteq M$. Such $A$ have the following property, that can be viewed as an analogue to
Tarski-Vaught Criterion for being an elementary submodel:

\begin{de}\label{6}\label{dfn:complete}
$A\subseteq {\cal C}$ is {\it complete} if for every
formula $\psi(\bar x,\bar y)$ and $\bar b\subseteq A, \models
(\exists\bar x\in P)\psi(\bar x,\bar b)$ implies $(\exists
\bar a\subseteq P\cap A)\models \psi(\bar a,\bar b)$. 
\end{de}

As we have noted above, the following is clear:

\begin{obs}\label{obs:complete}
If $M\prec {\cal C}$ and $P^M\subseteq A\subseteq M$, then
$A$ is complete.
\end{obs}

Under the assumption of saturation of $P^M$, the above also characterizes complete sets; see Proposition \ref{8} below (but we will not use this in this paper). 

\medskip

The following useful characterization offers another understanding of the notion of completeness (see Observation 4.2 in \cite{ShUs322a}):

\begin{obs}
\label{obs:complete_characterization}
A set $A$ is complete if and only if for every $\bar a\subseteq A$ and
$\phi(\bar x,\bar y)$  the $\phi$-type $tp_\phi (\bar a/P^{\cal C})$ is definable
over $A\cap P^{\cal C}$ and $A\cap P^{\cal C}\prec P^{\cal C}$.
\end{obs}

In fact, in terms of definability, one can say a bit more:

\begin{fact} \label{10}For any complete $A$ there are $\langle \Psi_\psi :\psi(\bar
x,\bar y)\in L(T)\rangle $ (depending on $A$) such that for all
$\bar a\subseteq A$, $tp_\psi (\bar a/P\cap A)$ is definable by
$\Psi _\psi (\bar y,\bar c)$ for some $\bar c\subseteq A\cap P$.
\end{fact}

\smallskip

It is useful to note that, since $T$ has QE, the property of completeness for a set $A$  depends only on its first order theory (as a substructure of $\cC$):

\begin{lem}(Lemma 4.5 in \cite{ShUs322a})\label{lem:complete_QE}\label{7.5}
\begin{itemize}
\item[(i)] If $A_1\equiv A_2$, then $A_1$ is complete iff $A_2$ is
complete.
\item[(ii)] $A$ is complete iff whenever the
sentence $$\theta=:(\forall \bar y)[S(\bar y)\leftrightarrow (\exists x\in P)
R(x,\bar y)]$$ for quantifier free 
$R,S$ is satisfied in ${\cal C}$, then
$A$ satisfies $\theta$.
%
%
\end{itemize}
\end{lem}

Let us also recall a few remarks on the relation between $P$ and the algebraic closure in complete sets.

\smallskip

Let us now recall the relevant notion of type for this context. Note that it is only defined over a complete set.

\begin{de}\label{6.5}\label{dfn:startypes}

Let $A$ be a complete set. 

\begin{itemize}

\item[(i)] Let
$$S_*(A)=\{tp(\bar c/A):P\cap (A\cup \bar c)=P\cap A {\rm\ and}\ A\cup \bar
c\ {\rm is\ complete}\}$$
\item[(ii)] $A$ is \emph{stable over $P$}, or simply \emph{stable}, if for all $A^\prime$
with $A^\prime\equiv A$, we have $|S_*(A^\prime)|\leq |A^\prime|^{|T|}$.
\end{itemize}
\end{de}

\begin{remark}
\begin{enumerate}
\item Even though ``stability over $P$'' is a more appropriate and accurate name for our notion of stability of a set (and the term ``stable set'' exists in literature, and has a different meaning), since we  have only one notion of stability in this article (stability over $P$), we will sometimes omit ``over P'' and simply write ``stable''. 
\item Sometimes we refer to types in  $S_*(A)$ as \emph{complete types over $A$ which are weakly orthogonal to $P$}. 
\end{enumerate}

\end{remark}

\begin{remark} (Remark 4.10 in \cite{ShUs322a})
	If $A$ is complete and $\c$ a tuple such that $tp(\c/A) \in S_*(A)$, then $\acl(A \cup \c)  \cap P^\cC = P^A$.
\end{remark}

The last remark also follows from the following more general criterion for a complete type being a *-type (see Observation 4.11 in \cite{ShUs322a}):

\begin{obs}\label{8.5}
 Let $A$ be a complete set and $\c$ a tuple. Then $tp(\c/A) \in S_*(A)$ if and only if for every formula $\psi(\x,\a,\y)$ over $A$ we have 
 \[ \models \exists \x \in P\, \psi(\x,\a,\c) \Longrightarrow \exists \bar b \in P\cap  A \, \text{such that} \, \models \psi(\bar b,\a,\c) \]
 \end{obs}
 
\medskip


The following Lemma (Lemma 4.12 in \cite{ShUs322a}) was at heart of many of our proofs in \cite{ShUs322a}, since it allows us to extend any ``small'' type to a type orthogonal to $P$, provided that $P$ is saturated.

\begin{lem}(\underline{The Small Type Extension Lemma})\label{extension}\label{9}\label{le:typeextension}
If $A \prec \cal C$, $A\cap P$ is $|M|$-compact, and $p(\bar x)$ is a partial type over $A$
of cardinality $<|A|$,  then there is some ${p^*}(\bar x)\in S_*(A)$
extending $p$.
\end{lem}

\medskip

This Lemma yields another characterization of completeness, justifying, in some sense, the original motivation behind this definition (see the discussion in the very beginning of this section). See Proposition 4.14 in \cite{ShUs322a}. 

\begin{pr}\label{approx}\label{8}\label{prp:complete_exnetdabletomodel}
Suppose that $A\cap P$ is $|A|$-compact. Then $A$ is complete if and only if
there exists an $M\prec {\cal C}$ with $P^M\subseteq A\subseteq M$.
If $|A|=|A|^{<|A|}>|T|$, we can add ``$M$ saturated''.
\end{pr}

\medskip

It is natural to ask, however, what happens if $P$ is not saturated, but, for instance, is only $\lam$-saturated, and $|p|<\lam$? We don't know the full answer to this question. However, the following ``poor man's version'' holds if we already know that $A$ can be extended to a $\lam$-saturated model $M\models T$ with $P^M = P^A$. We will make use of this version later. 

\begin{obs}(\underline{The Small Type Extension Lemma, Version II})\label{extension1}\label{obs:typeextension}
	Let $M \prec \cal C$ be $\lam$-saturated, and $p(\bar x)$ is a partial type over $M$ of cardinality $<\lam$ (where $\x$ is potentially an infinite tuple of cardinality $<\lam$), and let $A$ be such that $P^M \subseteq A \subseteq M$. Then there is some ${p^*}(\bar x)\in S_*(A)$
extending $p$.

\end{obs}

\begin{proof}
 First of all, note that since $A \subseteq M$ with $P^M = P^A$, clearly $A$ is complete (see Observation \ref{obs:complete}). By saturation, there is $\c\in M$ realizing $p$. As before, $B = A\cup\set{\c}$ is complete, and, of course, $P^B = P^A (= P^M)$. So $p^* = \tp(\c/A)$ is as required. 
\end{proof}

\smallskip

Note that the cardinality of $A$ does not matter; it could be that $|A| = |M| \ge \lam$.

\subsection{Stability and rank}

Next, let us recall a notion of rank that ``captures'' stability over $P$ (\cite{PiSh130}, \cite{ShUs322a}). 

\smallskip

\begin{de}\label{R}
For a complete set $A$, a (partial) $n$-type $p(\bar x)$ (with parameters in ${\cal C}$), 
  sets $\Delta _1,\Delta _2$ of formulas $\psi (\bar
x,\bar y)$, and a cardinal $\lambda$, we define when $R^n_A(p,\Delta
_1,\Delta _2,\lambda)\geq \alpha$. We usually omit $n$.

\begin{itemize}
\item[(i)] $R_A(p,\Delta _1,\Delta _2,\lambda)\geq 0$ if $p(\bar x)$ is consistent. 

\item[(ii)] For $\alpha$ a limit ordinal: $R_A(p,\Delta _1,\Delta
_2,\lambda)\geq \alpha$ if $R_A(p,\Delta _1,\Delta
_2,\lambda)\geq \beta$ for every $\beta <\alpha$.

\item[(iii)] For $\alpha =\beta +1$ and $\beta$ even:
For $\mu<\lambda$ and finite $q(\bar x)\subseteq p(\bar x)$ we can
find $r_i(\bar x)$ for $i\leq\mu$ such that;

\begin{itemize}
\item[{1.}] Each $r_i$ is a $\Delta _1$-type over $A$,

\item[{2.}] For $i\not= j, r_i$ and $r_j$ are explicitly contradictory
(i.e. for some $\psi$ and $\bar c$, $\psi (\bar x,\bar c)\in r_i,
\neg\psi(\bar x,\bar c)\in r_j$).

\item[{3.}] $R_A(q(\bar x)\cup r_i(\bar x), \Delta _1,\Delta
_2,\lambda)\geq \beta$ \red{for all $i$}.
\end{itemize}

\item[(iv)] For $\alpha =\beta +1: \beta$ odd: For 
$\mu<\lambda$ and finite $q(\bar x)\subseteq p(\bar x)$ and $\psi
_i\in \Delta _2, \bar d_i\in A$ ($i\leq \mu$), there are $\bar b_i\in
A\cap P$ such that $R(r_i,\Delta _1,\Delta
_2,\lambda)\geq \beta$ where $r_i=q(\bar x)\cup \{(\forall\bar
z\subseteq P) \left[\psi _i(\bar x,\bar d_i,\bar z)\equiv\Psi_{\psi
_i}(\bar z,\bar b_i)\right] \colon i<\mu\}$ where $\Psi_{\psi _i}$ is as in Fact \ref{10}.

\end{itemize}

$R^n_A(p,\Delta _1,\Delta _2,\lambda)= \alpha$ if $R^n_A(p,\Delta
_1,\Delta _2,\lambda)\geq \alpha$ but not $R^n_A(p,\Delta
_1,\Delta _2,\lambda)\geq \alpha +1$. $R^n_A(p,\Delta
_1,\Delta _2,\lambda)=\infty$ iff $R^n_A(p,\Delta
_1,\Delta _2,\lambda)\geq \alpha$ for all $\alpha$.
\end{de}

The main case for applications will be $\lambda =2$. Note that the
larger $R^n_A(p,\Delta _1,\Delta _2,\lambda)$, the more evidence there
is for the existence of many types $q(\bar x)\in S_*(A)$ consistent with
$p(\bar x)$.

\medskip 

See section 5 of \cite{ShUs322a} for a detailed discussion of some basic properties of the rank. But let us recall here some of the more important results:

\medskip

\begin{fact}\label{rank}(Fact 5.3 in \cite{ShUs322a})
%
For every $p$ there is a finite $q\subseteq p$, such that 
$R_A(p,\Delta _1,\Delta _2, 2)=R_A(q,\Delta _1,\Delta _2, 2)$.
%
%
\end{fact}

\begin{fact}\label{rankeven}(see Fact 5.3 in \cite{ShUs322a}). 
	Let $A$ be complete, $p\in S_*(A)$, $q^* \subseteq p$, and assume
	 $$R^n_A(q^*,\Delta _1,\Delta _2,\lambda)=R^n_A(p,\Delta _1,\Delta _2,\lambda)=k < \infty$$ Then $k$ is even.
\end{fact}

\begin{theorem}\label{13}\label{thm:stablerank}(see Theorem 5.4 in \cite{ShUs322a})

The following are equivalent:
\begin{itemize}
\item[(i)] $A$ is stable.
\item[(ii)] For every finite $\Delta _1$ and finite $n$ there are some
finite $\Delta _2$ and finite $m$ such that 
$R^n_A(\bar x=\bar x,\Delta _1,\Delta _2,2)\leq m$.
\end{itemize}
\end{theorem}

%

\begin{fact} (Corollary 5.5 in \cite{ShUs322a})
	In Definition \ref{6}(iv), it is not necessary to consider all $A' \equiv A$. 
	More specifically, a complete set $A$ is stable if and only if $|S_*(A')| \le |A'|^{|T|}$ for some $A' \equiv A$ saturated, $|A'|>|T|$. 
\end{fact}

%


\medskip

We  often omit the superscript and the subscript in the rank $R^n_A$, and write simply $R$ (at least when 
$n$ and $A$ are easily deduced from the context).

\medskip

It  follows (see  Corollary 5.6 in \cite{ShUs322a}) that every type $p \in S_*(A)$ over a stable set $A$ is definable internally in $A$:



\begin{co}\label{14} 

\begin{enumerate}
\item
If $A$ is stable, then for every $\psi (\bar
x,\bar y)\in L(T)$ there is $\Psi _\psi$ in $L(A)$ such that if $p\in
S_*(A)$, then for some $\bar b\subseteq A, \Psi _\psi (\bar
y,\bar b)$ defines $p| \psi$ in $\cC_A$. 

Specifically, for every $\c \in A$, $\psi(\x,\c) \in p$ if and only if $A \models \Psi _\psi (\bar
c,\bar b)$.
 \item Moreover, if $|A|\geq 2$, then for every $\psi(\x,\y)$, there is a definition $\Psi_\psi(\x,\y)$ as above which works uniformly for all $B \equiv A$ and $p \in S_*(B)$. 
 \end{enumerate}
\end{co}

\section{Stability and primary models}

In this section, we prove the existence of $\lam$-prime models. We will need the following version of stability over $P$:

\begin{de}\label{dfn:fullystable}
We say that a complete set $A$ is \emph{fully stable} over $P$ if any complete $B \subseteq A$ is stable over $P$. We say that $T$ is \emph{fully stable} over $P$ if all its complete sets are (fully) stable. 
\end{de}

\smallskip

Let us recall the definitions of the notions relevant for the discussion in this section. 

\begin{de}

A (partial) type $p$ over a set $A$ is called \emph{$\lam$-isolated} if there exists a subset $r \subseteq p$ with $|r|<\lam$ such that $r \equiv p$. 

So $p$ is isolated if it is $\aleph_0$-isolated. 
\end{de} 

\begin{de}
Let $N$ be a model, $P^N \subseteq B \subseteq N$.
\begin{enumerate}
\item We say that a model $N$ is \emph{$\lam$-prime} over a $B$ if $N$ is $\lam$-saturated, and it can be elementarily 
embedded over $B$ into any $\lam$-saturated model containing $B$.  

\item 

	We say that $N$ is \emph{$\lam$-atomic} over $B$ if 
for every $\bar d\subseteq N$, $tp(\bar d,B)$ is
$\lambda$-isolated over some $B_{\bar d}\subseteq B, |B_{\bar
d}|<\lambda$.
\item
	We say that the sequence $\d = \{d_i:i<\alpha\} \subseteq N$ is a $\lam$-construction over $B$ in $N$ if for all $i<\al$, the type $tp\{d_i/B\cup\{d_j:j<i\})$ is $\lambda$-isolated.

\item
	We say that a set $C \subseteq N$ is $N$ is \emph{$\lam$-constructible} over $B$ in $N$ if 
	there is a $\lam$-construction $\d$ over $B$ in $N$. 
	
	In particular, we say that $N$ is
	\emph{$\lam$-constructible} over $B$ if 
there is a construction $N=B\cup\{d_i:i<\al\}$ such that for all $i<\lam$ the type 
$tp\{d_i/B\cup\{d_j:j<i\})$ is $\lambda$-isolated. 
\item 
	We say that a model $N$ is \emph{$\lambda$-primary} over $B$ if it is $\lam$-constructible and $\lam$-saturated. 
\end{enumerate}
\end{de}

\begin{remark}\label{rem:primary}
\begin{enumerate}
\item
 If $N$ is $\lam$-primary over $B$, then it is $\lam$-prime over $B$.
\item ($\lam$ regular)
 If $N$ is $\lam$-constructible over $B$ witnessed by a construction $N=B\cup\{d_i:i<\lambda\}$, then for every $\al<\lam$, $\tp(\{d_i:i<\al\}/B)$ is $\lam$-isolated. Hence $N$ is $\lam$-atomic over $B$.
\end{enumerate}
\end{remark}
%

%

Recall that the Small Type Extension Lemma (version II) allows us to extend a ``small'' type over a complete set $A$ to a complete type over $A$ orthogonal to $P$. The following Lemma shows that under the assumption of stability over $P$ and $\lam$-saturation, this complete type may be assumed to be $\lam$-isolated.  

\begin{lem}\label{16} 
\begin{enumerate}
\item

Assume $M \models T$ is $\lam$-saturated, $B\subseteq M$ is stable, $P^B = P^M$.
Let  $p(\bar x)$  be
an $m$-type over $B$, $|p(\bar x)|<\lambda$. 

Then there is
$q(\bar x)$ such that $|q(\bar x)|\leq |T|$,
$p(\bar x)\cup q(\bar x)$ consistent, and there is $r\in S_*(B)$ such
that $p(\bar x)\cup q(\bar x)\equiv r(\bar x)$. 

In particular,  $r(\bar x)$ is
$\lambda$-isolated.
\item
The previous clause is also true if $\x$ is an infinite tuple with $<\lam$ variables, but in this case we can only 
require that  $|q|<\lam$.

Specifically, if $|\x| = \ka < \lam$, then there exists $|q| \le |T|\cdot\ka$ as above.

\end{enumerate}

\end{lem}

\begin{proof} 

\begin{enumerate}
\item 
Let $\{\psi _i(\bar x,\bar y_i):i<|T|\}$ list all
formulas of $L(T)$. Let $\Delta _i$ be finite such that $R(\bar x=\bar
x,\{\psi _i\},\Delta _i,2)<\omega$ (where $R = R^m_B$). Define $q_i(\bar x)$ by induction
on $i < |T|$ such that
\begin{itemize}
\item[(a)] $q_i$ is finite and is over $B$,
\item[(b)] $p(\bar x)\cup \bigcup _{j\leq i}q_j(\bar x)$ is
consistent, and
\item[(c)] $R(p\cup\bigcup _{j\leq i}q_j,\{\psi _i\},\Delta _i,2)$ is
minimal with respect to (a) and (b).
\end{itemize}

This is possible since $B$ is stable. 


By Observation \ref{extension1}, there is $p^* \in S_*(B)$ extending $p\cup\bigcup _{j\leq i}q_j$. 


Clearly, $R(p^*,
\{\psi _i\},\Delta _i,2) \le R(p\cup\bigcup _{j\leq i}q_j,
\{\psi _i\},\Delta _i,2)$, and for some finite $q' \subseteq p^*$ (see Fact \ref{rank}) we have $R(p^*,
\{\psi _i\},\Delta _i,2) = R(q',
\{\psi _i\},\Delta _i,2)$. If $R(p^*,
\{\psi _i\},\Delta _i,2) < R(p\cup\bigcup _{j\leq i}q_j,
\{\psi _i\},\Delta _i,2)$, setting $q'_i = q_i \cup q'$ would contradict the ``minimality'' of $q_i$ (clause (c) above). Hence $R(p^*,
\{\psi _i\},\Delta _i,2) = R(p\cup\bigcup _{j\leq i}q_j,
\{\psi _i\},\Delta _i,2)$. 


By Fact \ref{rankeven}, $R(p^*,
\{\psi _i\},\Delta _i,2)$ is even, hence so is $R(p\cup\bigcup _{j\leq i}q_j,\{\psi _i\},\Delta _i,2)$.
In particular, by the definition of the rank, we have: 

\begin{itemize}
\item[(d)]
For no $\bar b\subseteq B$ do we have $R(p\cup\bigcup _{j\leq
i}q_j\cup \{\pm\psi_i(\bar x,\bar b)\},\{\psi _i\},\Delta
_i,2)\geq R(p\cup\bigcup _{j\leq i}q_j,\{\psi _i\},\Delta _i,2)$
\end{itemize}

Clearly $q = |\bigcup _{j\leq |T|}q_j|\leq |T|$. 

By Observation \ref{extension1} there is some $r\in S_*(B)$ (realized in $M$)
such that $p\cup\bigcup_{j<|T|}q_j\subseteq r$. By (c) and (d) above it follows that $p\cup\bigcup_{j<|T|}q_j\vdash r$.

\item 
Let $\seq{\psi_i(\x_i, \y_i)\colon i<\ka}$ list all the formulas where $\x_i$ is a finite tuple from $\x$ (so $|T| \le \ka < \lam$), and define $q_i$ on induction on $\ka$ just as in the proof of the previous clause. Since each $q_i$ is finite, $q = |\bigcup _{j\leq |\ka|}q_j|\leq \ka<\lam$, and we can again use Observation   \ref{extension1} in order to obtain $r\in S_*(B)$ as required.
\end{enumerate}

\end{proof}

We are now ready to prove the main result of this paper. 

\begin{theorem}\label{th:primary}
	Assume that $M \models T$ is \lam-saturated and fully stable over $P$, $B \subseteq M$,
	$P^B = P^M$ (so in particular $P^B$ is $\lam$-saturated). 
	
	 Then there is $N \supseteq B$, $N \subseteq M$, $N$ is $\lam$-primary over $B$. 
\end{theorem}
\begin{proof}
Let $\ka = |M|$. Construct by induction on $i$ a sequence $\lseq{\d}{i}{\delta}$ (where $\delta \le \ka$ is a limit ordinal) of sequences such that:

 \begin{enumerate}
\item 
	$\d_i = \lseq{d}{\al}{\al_i} \subseteq M$ (where $d_\al$ is a singleton) 
\item
	$\d_i$ is increasing and continuous with $i$ (so in particular $\al_i \le \al_j$ for $i<j)$
\item 
	For every $\al$, the type $\tp(d_\al/B \cup \set{d_\be\colon \be<\al})$ is $\lam$-isolated
\item
	For every $i < \delta$, every $A \subseteq B\cup\set{\d_i}$, $|A|<\lam$, every type over $A$ is realized by some $d_\al$

\end{enumerate}

We shall show that, if we succeed to carry out the construction at every stage, then for some limit ordinal $\delta \le \ka$, $\d_\delta = \bigcup_{i<\delta} \d_i$ is a $\delta$-construction of a $\lam$-saturated model $N$ over $B$.


\bigskip

Let $\d_0 = \seq{}$.

For $i=\delta$ limit, let $\d_\delta = \bigcup_{j<\delta}\d_j$. If $B_\delta = B \cup \d_\delta$ is a $\lam$-saturated model of $T$, then we are done. Note that this is in particular the case if $B_\delta = M$, since $M$ is $\lam$-saturated. If not, there is a type over a subset of $B_\delta$ of size $<\lam$ that is not realized in $B_\delta$ (but is realized in $M$), and we continue the construction. 

For $i = j+1$,  let $B_j = B \cup \d_j$. Let the sequence $\seq{p_{j,\ga} \colon \ga \in [j,\mu)}$ list all the types over subsets of $B_j$ of cardinality $<\lam$. Note that since $M$ is $\lam$-saturated, we have $\mu \le \ka = |M|$.


Now consider the sequence 
$\seq{p_{\ell,j}\colon \ell<i}$. 

Recall that by induction, $\d_j = \lseq{d}{\al}{\al_j}$.  For $\ell<i$ and $\al = \al_j + \ell$, let $d_\al$ realize $p_{\ell,j}$ such that $\tp(d_\al/B \cup \set{d_\be\colon \be<\al}) \in S_*(B \cup \set{d_\be\colon \be<\al})$ is $\lam$-isolated (this is possible by Lemma \ref{16}; note that, since $M$ is fully stable over $P$, the set $B \cup \set{d_\be\colon \be<\al}$ is stable).

Clearly, setting $\al_i = \al_j + i$, the sequence $\d_i = \lseq{d}{\al}{\al_i}$ is as required. 


\end{proof}

\begin{remark}
	The assumption of full stability is not unreasonable. For example, it holds in $ACFA_0$ over $P = Fix(\sigma)$ for any model $M$ with a $\lam$-saturated $P$-part (where $\lam$ is uncountable or $\aleph_\epsilon$). This is easy to see a posteriori from the main result of \cite{Chatzidakis2020RemarksAT} in combination with Theorem 7.3 in \cite{ShUs322a}; but it can also be worked out directly by counting types almost orthogonal to $P$ over the relevant sets in $ACFA_0$. 
	
	At the same time, we believe that this assumption is not necessary, and it would be enough to assume stability for a certain kind of $n$-diagrams (as, e.g., in the last chapter of \cite{Sh:c}). The theory of finite diagrams in this context is currently under development, and we are hoping to return to this question in a future work. 
\end{remark}

Let us state explicitly a consequence of our results that is easy to phrase:

\begin{co}\label{cor:main}
	Let $T$ be fully stable over $P$, $A$ be a complete set with $P^A$ $\lam$-saturated for some $\lam>|T|$. Assume that the class of models 
	$$\mathcal{K} = \set{N\models T \colon P^N = P^A, N \text{ is } \lam-\text{saturated}}$$
	is non-empty. Then $\mathcal{K}$ has a prime member over $P$: there exists $N_0 \in \mathcal{K}$, which is 
	elementarily embeddable into any $N \in \mathcal{K}$ over $P^A$. 
\end{co}

\bibliography{common.bib}
\bibliographystyle{alpha}

\Addresses

\end{document}